\documentclass[a4paper,11pt]{amsart}  

\makeatletter

\def\blfootnote{\gdef\@thefnmark{}\@footnotetext}

\makeatother

 \usepackage{amsaddr}
\usepackage{graphicx}
\usepackage{color}
\usepackage{pst-grad}
\usepackage{url}
\usepackage{pstricks}
\usepackage[utf8]{inputenc}
\usepackage{pstricks-add}
\usepackage{pgf}
\usepackage{tikz-cd}
\usepackage{tikz}
\usepackage{pst-all}

\usepackage[thinlines]{easytable}
\usepackage[all]{xy}
\usepackage{pinlabel}
\usepackage{psfrag}
\usepackage[all]{xy}
\usepackage{hyperref}

\usepackage{amsmath, amsfonts, amsthm, amssymb, amscd, mathtools, tensor}
 \newtheorem{thm}{Theorem}[section]
 \newtheorem{prop}[thm]{Proposition}
 \newtheorem{lem}[thm]{Lemma}
 \newtheorem{cor}[thm]{Corollary}
\theoremstyle{definition}
 \newtheorem{rem}[thm]{Remark}
 \newtheorem{definition}[thm]{Definition}
\numberwithin{equation}{section}

\def\d#1{{#1\kern-0.4em\char"16\kern-0.1em}}
\def\D#1{{\raise0.2ex\hbox{-}\kern-0.4em#1}}
\def \Dj{\mbox{\raise0.3ex\hbox{-}\kern-0.4em D}}
\definecolor{britishracinggreen}{rgb}{0.0, 0.26, 0.15}
\def\zn{,\kern-0.09em,}

\newrgbcolor{zzttqq}{0.8 0.8 0.8}
\newrgbcolor{svetla}{.9 .9 .9}
\newrgbcolor{plava}{.67 .84 .9}
\newrgbcolor{zelena}{.56 .93 .56}
\newrgbcolor{zelena1}{.32 .53 .32}

\title{From Heegaard diagrams to surgery}

\date{}
\begin{document}
\author[J. Nikoli\' c, V. Ovaskainen and Z. Petri\'c]{J. Nikoli\' c\textsuperscript{1,}$^{\ast}$, V. Ovaskainen\textsuperscript{1} and Z. Petri\' c\textsuperscript{2}}
\address{\scriptsize{\textsuperscript{1}Faculty of Mathematics, University of Belgrade, Studentski trg 16, 11158 Belgrade, Serbia\\
e-mails: jovana.nikolic@matf.bg.ac.rs, vukovaskainen@gmail.com}}
\address{\scriptsize{\textsuperscript{2}Mathematical Institute SANU, Knez Mihailova 36, p.f.\ 367, 11001 Belgrade, Serbia\\
e-mail: zpetric@mi.sanu.ac.rs}}
\thanks{$^\ast$Corresponding author.}
\begin{abstract}
The precise steps of a procedure of going from Heegaard diagrams to framed
link diagrams are introduced in this note.
\end{abstract}
\blfootnote{ {\it Mathematics Subject Classification} ({\it
        2020}): 57-01, 57K30.}
\blfootnote{{\it Key words and phrases}: 3-manifolds, Kirby's calculus, Heegaard diagram, surgery presentation.}
\maketitle

\section{Introduction}

By a \emph{surface} we mean here a closed, connected and orientable
surface and a \emph{manifold} is a closed oriented three-dimensional
manifold. (Homeomorphisms between manifolds are orientation preserving.) Such manifolds are presentable in terms of diagrams in several
ways, and we mention here the presentation in terms of Heegaard diagrams
and in terms of framed links in $S^3$. The latter is usually called
\emph{surgery presentation}. The aim of this note is to provide a precise
steps in transforming planar Heegaard diagrams into a surgery. The details of a
reverse transformation are given, for example, in \cite{A17}.

We are aware that this procedure is considered as a folklore and
that it could be traced back to Lickorish's results from \cite{L62}.
However, when we were faced with a concrete problem of finding surgery
presentation for some small covers (see \cite{DJ91} or \cite{BP15} for
this notion), a precise algorithm for doing this was not reachable in the
literature. A companion to the algorithm provided in this note is a simple
example of a small cover obtained by an identification of the facets of
eight tetrahedra. This identification gives immediately a (demanding)
Heegaard diagram of the obtained manifold (the real projective space), and
this diagram could help one to follow the steps of the algorithm. Besides
the precision in defining curves and Dehn twists, we ought to point out
some heuristics that could considerably shorten the procedure of
transforming Heegaard diagrams into a surgery and to produce a result of
lower complexity.

\section{Heegaard diagrams, surgery and Dehn twists}\label{heegaard}

We start by fixing the terminology and the notation.

\begin{definition}
Let $\Gamma$ be a surface of genus $g$. A \emph{system of attaching
circles} for $\Gamma$ is
a set $\{\gamma_1,\ldots,\gamma_g\}$ of simple closed curves in this
surface such that:
\begin{enumerate}
\item the curves $\gamma_i$ are mutually disjoint,
\item $\Gamma-\gamma_1-\ldots-\gamma_g$ is connected.
\end{enumerate}
Let $\alpha$ and $\beta$ be two systems of attaching circles for $\Gamma$.
The triple $(\Gamma,\alpha,\beta)$ is a \emph{Heegaard diagram} of \emph{genus} $g$ (see
\cite{PS}, \cite{OS06} and \cite{J07}).
\end{definition}

\begin{rem}\label{rem1}
Let $\Gamma$ and $\Delta$ be two surfaces of genus $g$. If
$\{\gamma_1,\ldots,\gamma_g\}$ and $\{\delta_1,\ldots,\delta_g\}$ are
systems of attaching circles for $\Gamma$ and $\Delta$, respectively, then
there exists a homeomorphism from $\Gamma$ to $\Delta$, such that for
every $1\leq i\leq g$ it maps $\gamma_i$ to $\delta_i$. If $\Gamma$ and
$\Delta$ are oriented, one can choose this homeomorphism to be either
orientation preserving, or orientation reversing.
\end{rem}

\begin{definition}\label{meridian}
Let $H$ be a handlebody with $g$ handles. A \emph{meridional disk} is a
properly embedded disk in $H$ (its boundary belongs to the boundary of
$H$) such that by removing its neighbourhood from $H$ one obtains a
handlebody with $g-1$ handles. A \emph{complete system of meridional
disks} for $H$ consists of $g$ mutually disjoint disks such that by
removing their neighbourhoods from $H$ one obtains a ball.
\end{definition}

\begin{rem}\label{rem2}
The boundaries of a complete system of meridional disks for $H$ make a
system of attaching circles for $\partial H$. We call this system
\emph{meridians} of $H$.
\end{rem}

From Remarks \ref{rem1} and \ref{rem2} we conclude the following.

\begin{cor}\label{posledica4}
Let $\{\gamma_1,\ldots,\gamma_g\}$ be a system of attaching circles for a
connected, orientable surface $\Gamma$. Let $H$ be a handlebody with $g$
handles and $\{\mu_1,\ldots,\mu_g\}$ its complete system of meridional
disks. Then there exists a homeomorphism
$h\colon\Gamma\stackrel{\approx\;\:}{\to}\partial H$ such that for every
$1\leq j\leq g$ it maps $\gamma_j$ to~$\partial \mu_j$.
\end{cor}

\begin{definition}
Let $M$ be a manifold and let $\Gamma$, $H$ and $H'$ be respectively a
surface, and two handlebodies embedded in $M$ such that $\partial
H=\partial H'=\Gamma=H\cap H'$ and $M=H\cup H'$. The triple
$(\Gamma,H,H')$ is a \emph{Heegaard splitting} of $M$.
\end{definition}

\begin{definition}
A Heegaard diagram $(\Gamma,\alpha,\beta)$ is \emph{compatible} with a
Heegaard splitting $(\Gamma,H,H')$ of a manifold, when $\alpha$
consists of meridians of $H$, while $\beta$ consists of meridians of $H'$.
\end{definition}

\begin{definition}
A Heegaard diagram \emph{presents} a manifold, when there exists a
Heegaard splitting of this manifold, which is compatible with this
diagram.
\end{definition}

The following remark is analogous to \cite[IV, Section~10.1,
Theorem~10.2]{PS} (see also \cite[Remark after Definition~2.3]{OS06}).

\begin{rem}\label{jedinstvenost}
Every Heegaard diagram presents a unique manifold.
\end{rem}
\begin{rem}\label{s3 diagram}
A sufficient condition for $(\Gamma,\alpha,\alpha')$, where $\Gamma$ is of
genus $g$, to present the sphere $S^3$ is that $\Gamma$ is a connected sum
$T^2_1\sharp\ldots\sharp T^2_g$ of $g$ copies of tori, such that
$\alpha_i$ and $\alpha'_i$ are two simple closed curves in $T^2_i$ that
intersect transversally in a single point. This condition is equivalent
to: for $i\neq j$, $\alpha_i\cap \alpha'_j=\emptyset$, and for every $i$,
$\alpha_i$ intersects transversally $\alpha'_i$ in a single point.
\end{rem}

\begin{definition}\label{proper emb}
Assume that $\Gamma$, $\alpha$ and $\alpha'$ satisfy the conditions of Remark~\ref{s3 diagram} and $H$, $H'$ are two copies of a handlebody with $g$ handles such that $H$ is standardly
embedded in $\mathbb{R}^3\subset S^3$ and the complement of its interior with respect to $S^3$ is $H'$. The identification of $\Gamma$ with $\partial H=\partial H'$ so
that $\alpha$ become meridians of $H$, while $\alpha'$ become meridians of $H'$ is called a \emph{proper embedding} of $\Gamma$ in $S^3$ with respect to $(\alpha,\alpha', H, H')$.
\end{definition}

\begin{rem}\label{HegardHom}
Let $(\Gamma,\alpha,\beta)$ be a Heegaard diagram of genus $g$ that
presents a manifold $M$ and let the curves $\alpha'$ be introduced in this diagram according to Remark~\ref{s3 diagram}. Consider a proper embedding of $\Gamma$ with respect to $(\alpha,\alpha',H,H')$.
If $h\colon \Gamma\to\Gamma$ is an orientation preserving homeomorphism that for every $1\leq i\leq g$ maps $\alpha'_i$ to a curve isotopic to
$\beta_i$, then $M$ is homeomorphic to the manifold $((H\times \{0\})\cup
(H'\times\{1\}))/\sim$, where $\sim$ is such that for every $x\in\Gamma$ we have $(h(x),0)\sim (x,1)$.
\end{rem}

We present Heegaard diagrams in the following form. By cutting $\Gamma$
along the curves from $\alpha$, one obtains a ``sphere with $2g$ holes''.
In this way every such curve appears as the boundary of two discs removed
from the sphere. The curves from $\beta$ intersecting the curves from
$\alpha$ are presented by paths connecting some points on the boundaries
of the removed discs. In this way, a Heegaard diagram is presented as a
part of the sphere (projected to the plane) and is called \emph{planar}.
Such a diagram should provide an unambiguous identification of the
endpoints of paths belonging to $\beta$. A system of attaching curves
$\alpha'$ is introduced in this diagram according to Remark~\ref{s3
diagram}.

A \emph{surgery} is a link in $S^3$ whose components are labeled by
integers (framings). A surgery is interpreted as a manifold by removing a
tubular neighbourhood (a solid torus) of each link component and sewing it
back according to the corresponding framing $n$---a meridian of the solid
torus is sewed back so that it winds $n$ times in the direction of the
meridian and once in the direction of the longitude (see \cite{PS}).

\begin{definition}\label{twist}
Let $l$ be a simple closed curve in an oriented surface $\Gamma$, and let
$A$ be an annulus embedded in $\Gamma$ with $l$ as one boundary component.
We define intuitively and up to isotopy the \emph{right Dehn twist}
$\theta_l\colon\Gamma\to\Gamma$ \emph{along} $l$. It is the identity in
the complement of $A$, and inside $A$ it takes a path transverse to the
core of $A$ to a path that wraps once around $A$ turning in the right-hand
direction (with respect to the orientation of $\Gamma$). It is evident
that $\theta_l$ is a homeomorphism and its inverse $\theta^{-1}_l$ is
called the \emph{left Dehn twist along} $l$. Both left and right Dehn
twists are illustrated in Figure~\ref{dehn twist} (see \cite[Introduction]{L64} for a more formal definition of [Dehn] twist).
\end{definition}

\begin{figure}[h!h!h!]
    \centering
    \includegraphics[width=1\textwidth]{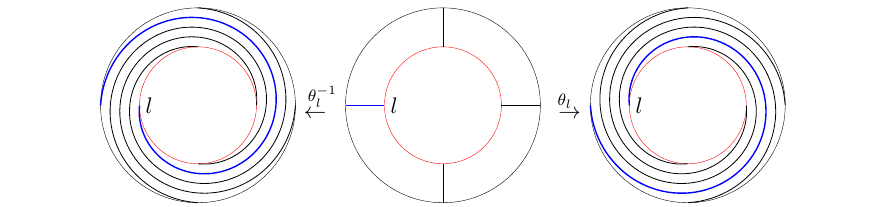}
    \caption{The left and the right Dehn twists}\label{dehn twist}
\end{figure}

\begin{definition}
Let $l$ and $l'$ be two oriented simple closed curves on $\Gamma$ that intersect transversally. The sign of an intersection point of $l$ with $l'$ is +1 when a pair of tangent vectors of $l$ and $l'$ in this order is consistent with an oriented basis for $\Gamma$, otherwise the sign is -1. The \emph{algebraic intersection number} $i_A(l,l')$ of $l$ with $l'$ is the sum of signs of intersection points of $l$ with~$l'$.
\end{definition}

\section{The algorithm}

Let $M$ be presented by a Heegaard diagram $(\Gamma,\alpha,\beta)$ in
which a family of curves $\alpha'$ is introduced according to Remark~\ref{s3
diagram}, so that $(\Gamma,\alpha,\alpha')$ presents $S^3$. Consider a proper embedding of $\Gamma$ with respect to $(\alpha,\alpha',H,H')$.

The first step in transforming this Heegaard diagram into a surgery for $M$  consists in finding a
composition of Dehn twists of $\Gamma$ that maps curves
$\alpha'_1,\ldots,\alpha'_g$ to the curves isotopic to
$\beta_1,\ldots,\beta_g$, respectively. This could be done along the lines
of the first part of the proof of \cite[Theorem~1]{L62}. In order to be
precise in determining right or left Dehn twists (the right twists have
priority in this note) assume that $\Gamma$ is oriented so that its normal
is directed towards the handlebody $H'$.

Our algorithm for finding a composition of Dehn twists of $\Gamma$ that
maps curves $\alpha'_1,\ldots,\alpha'_g$ to the curves isotopic to
$\beta_1,\ldots,\beta_g$ goes as follows.

\noindent Auxiliary steps ($p$ and $q$ are simple closed curves in
$\Gamma$ and for (D2b) and (D3) we assume that $p$ is oriented in an
arbitrary way):
\begin{itemize}
\item[(D1)] If $p$ intersects $q$ in a single point, then
$\theta_p\circ\theta_q$ maps $p$ to a curve isotopic to $q$. (See
\cite[Lemma~1]{L62}.)
\item[(D2a)] If $q$ belongs to a system of attaching circles and $p$
intersects $q$ twice, with zero algebraic intersection, then there is an
isotopy moving $p$ away from $q$ and leaving all the other curves from the
system fixed. (See the second paragraph of the proof of
\cite[Lemma~3]{L62}.)
\item[(D2b)] Let $A$ and $B$ be two intersection points of $p$ and $q$,
which are adjacent on $q$. We denote by $[A,B]$ the segment on $q$ having
no other common points with $p$. Assume that $p$ is oriented in the same
direction at $A$ and $B$ with respect to $q$ and by walking along $p$ in
its direction one has to turn \emph{left} at $A$ in order to continue the
walk along $[A,B]$.
    Then consider the curve $p_1$ that consists of the part of $p$
incoming to $A$, the segment $[A,B]$, and the part of $p$ outgoing
from $B$. (See \cite[Lemma~2, Case~1]{L62}.) The curve
$\theta_{p_1}[p]$ is isotopic to a curve intersecting $q$ in less
points than $p$.
\item[(D3)] Let $A$, $B$ and $C$ be three intersection points of $p$ and
$q$, which are adjacent on $q$ and the direction of $p$ induces the
circular order $ABC$. We denote by $[A,C]$ the segment on $q$ containing
$B$. Assume that $p$ is oriented in alternating directions at $A$, $B$ and
$C$ with respect to $q$ and by walking along $p$ in its direction one has
to turn \emph{left} at $A$ in order to continue the walk along $[A,C]$.
Then consider the curve $p_1$ that consists of the part of $p$ incoming to
$A$, the segment $[A,C]$, and the part of $p$ outgoing from $C$. (See
\cite[Lemma~2, Case~2]{L62}.) The curve $\theta_{p_1}[p]$ is isotopic to a
curve intersecting $q$ in less points than $p$.
\end{itemize}

The isotopies mentioned in (D1), (D2b) and (D3) do not move points outside
narrow neighbourhoods of $p$ and $q$, so that the curves disjoint from $p$
and $q$ remain fixed.

\noindent The algorithm:

\begin{itemize}
\item[(A1)] Suppose that for a fixed $0\leq t<g$, we have a composition $\delta$ of Dehn twists such that for $1\leq i\leq t$, $\delta[\alpha'_i]$ is isotopic to $\beta_i$.
\item[(A2)] Find a curve $\beta_{t+1}'=\delta[\alpha'_{t+1}]$. Since $\alpha'_1,\ldots,\alpha'_t,\alpha'_{t+1}$ are mutually disjoint, $\beta_{t+1}'$ is disjoint from $\beta_1,\ldots,\beta_t$.
\item[(A3)] If $\beta_{t+1}'$ is isotopic to $\beta_{t+1}$, then we go back to (A1) with $t=t+1$ and $\delta=\delta$.
\item[(A4)] If $\beta_{t+1}'$ does not intersect and is not isotopic to $\beta_{t+1}$, then try to find a curve $l$, or two curves $l$ and $m$, disjoint from all $\beta_i$, for $1\leq i\leq t$, such that in the sequence $\beta_{t+1}',l,\beta_{t+1}$ or $\beta_{t+1}',l,m,\beta_{t+1}$ the neighbours intersect in just one point and there are no other intersection points.
Then iterating (D1), one finds a composition $\delta'$ of Dehn twists that maps $\beta_{t+1}'$ to a curve isotopic to $\beta_{t+1}$, leaving all $\beta_i$, for $1\leq i\leq t$ fixed. We go back to (A1) with $t=t+1$ and $\delta=\delta'\circ \delta$.

If there are no such $l$ and $m$, then by using steps (D1), (D2a), (D2b) and (D3), one finds a composition $\delta''$ of Dehn twists mapping $\beta_{t+1}'$ into a curve isotopic to one that intersects less $\beta$-curves than $\beta_{t+1}'$. Denote this curve by $\beta_{t+1}''$. Moreover, $\delta''$ leaves all $\beta_i$, for $1\leq i\leq t$, fixed. Then go back to the beginning of (A4) with $\beta_{t+1}'=\beta_{t+1}''$ and $\delta=\delta''\circ\delta$.

We cannot remain forever in (A4) since the procedure from the preceding paragraph must eventually reach the point when $\beta_{t+1}'$ intersects no $\beta$-curve. This enables one to reason as in the third paragraph of the proof of \cite[Theorem~1]{L62} in order to find two curves $l$ and $m$ satisfying the conditions listed in the initial paragraph of~(A4).
\item[(A5)] If $\beta_{t+1}'$ intersects $\beta_{t+1}$, then by using steps (D2a), (D2b) and (D3), one finds a composition $\delta'$ of Dehn twists mapping $\beta_{t+1}'$ into a curve isotopic to $\beta_{t+1}''$ that either does not intersect $\beta_{t+1}$ or intersects it in a single point. In the first case go back to (A3), with $\beta_{t+1}'=\beta_{t+1}''$ and $\delta=\delta'\circ\delta$. In the second case go back to (A1) with $t=t+1$ and $\delta=\delta''\circ \delta'\circ \delta$, where $\delta''$ is a composition of Dehn twists obtained by (D1), which maps $\beta_{t+1}''$ to a curve isotopic to $\beta_{t+1}$. It is evident that this composition and the isotopy leave all $\beta_i$, for $1\leq i\leq t$ fixed.
\end{itemize}

After finding a composition $\theta_{\delta_n}\circ\ldots\circ\theta_{\delta_1}$ of Dehn twists that
maps $\alpha'$-curves onto $\beta$-curves, the next step in our procedure
is to build a link in $S^3$ out of the curves $\delta_1,\ldots,\delta_n$. This link consists of \emph{parallel copies} of $\delta_1,\ldots,\delta_n$ in the handlebody $H$ (see the proof of Proposition~\ref{proof}). This is done by
making a parallel copy of $\Gamma$ shallow inside $H$ (or ``digging the
trench''), layer by layer, so that $\delta_n$ is the deepest. In order to calculate the framing of each component of the link we rely on the following remark (that stems from the proof of Proposition~\ref{proof}) and lemma.
\begin{rem}\label{rem4.1}
Let us consider the (right) Dehn twist $\theta_l$ with respect to a curve $l$ on
$\Gamma$ whose parallel copy in $H$ makes a component $l'$ of the link.
Let $k\in \mathbb{Z}$ be the self-linking of $l\subseteq\Gamma$, i.e.\ the
linking number of $l$ and $l'$, assuming they are codirected. Then the framing of $l'$ is $k+1$. (It
should be noted that this $+1$ comes from the fact that $\theta_l$ is a
right Dehn twist with respect to the orientation of $\Gamma$ whose normal
is directed outside of $H$---if one switches from right to left twists,
i.e.\ to $\theta_l^{-1}$, then the framing of $l'$ is $k-1$, but there is
no need for this here.)
\end{rem}

For the following lemma we assume that for every $1\leq j\leq g$, the curves $\alpha_j$ and $\alpha'_j$ are oriented so that the intersection number $i_A(\alpha_j,\alpha'_j)$ is $-1$. Moreover, the embedding of $\Gamma$ is proper with respect to $(\alpha,\alpha',H,H')$.

\begin{lem}\label{lemma}
     Let $(\Gamma,\alpha,\alpha')$ be a Heegard diagram of genus $g$ for $S^3$, and $l$ be an arbitrarily oriented simple closed curve in $\Gamma$. The self-linking number of $l$ is equal to
     \begin{equation}\label{selflinking}
     \sum_{j=1}^g i_A(l,\alpha_j)\cdot i_A(l,\alpha'_j).
     \end{equation}
\end{lem}

\begin{proof}
Note that the self-linking of $l$ is equal to the linking number of $l$ and its parallel, codirected copy in $\Gamma$. Let us assume that our planar Heegaard diagram looks like one in Figure~\ref{fig:standardS3} (which is possible according to Remark~\ref{s3 diagram} with a help of an isotopy). We can also assume that $l$ meets $\alpha$ and $\alpha'$-curves transversely.

\begin{figure}[h!h!h!]
\centering
\includegraphics[width=.7\textwidth]{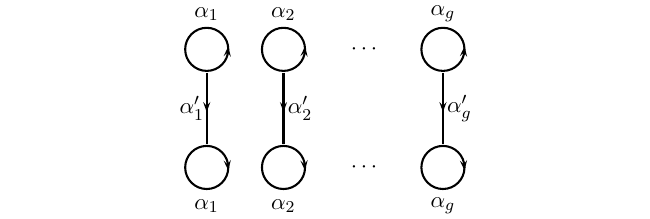}
\caption{Standard planar Heegaard diagram of $S^3$}\label{fig:standardS3}
\end{figure}

Note that apart of neighbourhoods of the pairs of curves $\alpha_j$, $\alpha'_j$ there is no contribution to the self-linking number of $l$. We can observe each of these neighbourhoods separately. Take a look at the leftmost picture in Figure~\ref{handle2}, which illustrates a neighbourhood of $\alpha_j$ and $\alpha'_j$ and a pair of intersection points---one of $l$ and $\alpha_j$ and the other of $l$ and $\alpha'_j$ (negative and positive intersection points are marked in black and red, respectively). It is evident that this pair of intersection points contributes by $-1$ to the sum (\ref{selflinking}).

\begin{figure}[h!h!h!]
    \centering
    \includegraphics[width=.3\textwidth]{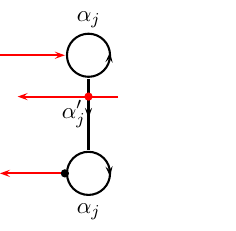} \includegraphics[width=.3\textwidth]{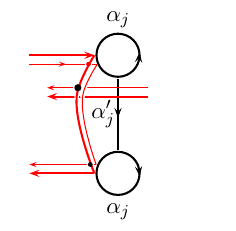} \includegraphics[width=.3\textwidth]{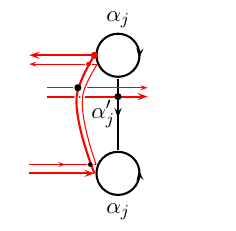}
    \caption{Curve $l$ meets the $j-$th handle}\label{handle2}
\end{figure}

On the other hand, this pair contributes to the self-linking of $l$ again by $-1$, which is justified by the picture in the middle of Figure~\ref{handle2} (parallel, codirected copies are added to the diagram as well as a ``bridge'' connecting two copies of the intersection point of $l$ and $\alpha_j$, which is ``above'' the diagram). The rightmost picture in Figure~\ref{handle2} shows that the orientation of $l$ does not affect the score. By summing over all these pairs, we get the desired result.
\end{proof}

We are ready to transform a given planar Heegaard diagram into a surgery through the following steps:

\begin{itemize}
\item[(S1)] Introduce the $\alpha'$-curves in the diagram, as well as all the curves involved in the composition $\delta=\theta_{\delta_n}\circ\ldots\circ\theta_{\delta_1}$ obtained by the algorithm.
\item[(S2)] Erase one copy of each $\alpha$-curve (assuming that $\alpha$ and $\alpha'$-curves are positioned as in Figure~\ref{fig:standardS3}, erase the bottom row of the copies of $\alpha$-curves). Moreover, erase all the curves from the diagram not involved in the composition $\delta$. Then introduce  ``bridges'' for every curve intersecting $\alpha$-curves. Bridges are disjoint arcs connecting the two endpoints of curves, which become arcs once we delete one of the $\alpha$-curves. These arcs are ``above'' every other curve in the configuration.
\item[(S3)] By applying Lemma~\ref{lemma} and Remark~\ref{rem4.1} introduce the framing tied to every curve involved in $\delta$.
\item[(S4)] The family of curves obtained by the previous steps does not represent a projection of a link since some intersections are not transversal and except for the ``bridges'', under and over crossings are not determined. Our (S4) transforms, step by step, this family of curves into a framed link: fix the curve $\delta_1$ and make a parallel copy of $\delta_2$ in $H$. This completely determines the under and over crossings of $\delta_2$ with the other curves present at the picture (see the left hand side of Figure~\ref{HDa}, where $\delta_1$ and $\delta_2$ are denoted by 1 and 2, respectively). Then repeat this with $\delta_3$ and so on. Note that in some cases the curves obtained by (S1)-(S3) produce multiple parallel copies. The example in Section~\ref{exmpl} provides a detailed illustrations of this step.
\end{itemize}

The following proposition is just a reformulation of \cite[Theorem~2]{L62}.

\begin{prop}\label{proof}
The surgery obtained by our algorithm results in a manifold homeomorphic to the manifold presented by the Heegaard diagram.
\end{prop}

\begin{proof}
Let $M$ be the manifold presented by a Heegaard diagram $(\Gamma,\alpha,\beta)$. Consider a proper embedding of $\Gamma$ with respect to $(\alpha,\alpha',H,H')$ and let $h\colon \Gamma\to\Gamma$ be an orientation preserving homeomorphism that for every $1\leq i\leq g$ maps $\alpha'_i$ to a curve isotopic to
$\beta_i$. Our manifold $M$ is
\[
H\cup_h H'=((H\times \{0\})\cup
(H'\times\{1\}))/\sim,
\]
where $\sim$ is such that for every $x\in\Gamma$ we have $(h(x),0)\sim (x,1)$.

\begin{figure}[h!h!h!]
    \centerline{\includegraphics[width=1\textwidth]{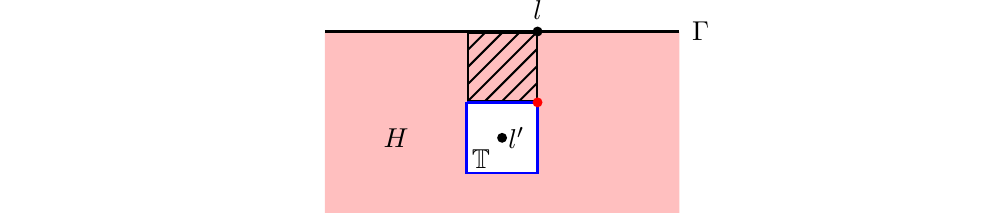}}
    \caption{A cross section of the handlebody $H$ normal to the curve $l$}
\label{digging}
\end{figure}

Assume for a moment that $h$ is just a Dehn twist $\theta_l$. Let $l'$ be a parallel copy of $l$ shallow inside $H$. Make a tubular neighbourhood $\mathbb{T}$ of $l'$ and remove it from $H$ (see Figure~\ref{digging}). One can easily extend $h$ to a homeomorphism $\chi\colon H-\mathbb{T}\to H-\mathbb{T}$, which is the identity in $H-\mathbb{T}$, except in the shaded region above $\mathbb{T}$ (see Figure~\ref{digging}), where it is defined in every layer as $\theta_l$. By relying on a straightforward topological argument (see for example \cite[Corollary~2.2]{FGOP}) one concludes that $(H-\mathbb{T})\cup_h H'$ is homeomorphic to $$(H-\mathbb{T})\cup_{\chi^{-1}\circ h} H'=(H-\mathbb{T})\cup H'=S^3-\mathbb{T}.$$

All that remains is to glue back the solid torus $\bar{\mathbb{T}}$ at the right position. For this, only important thing is how we glue a neighbourhood of a meridional disk of $\bar{\mathbb{T}}$. Assume that the blue square in Figure~\ref{digging} is the meridian of $\bar{\mathbb{T}}$. Then its left, right, and bottom side could be glued back by identity, while the top side must be glued to the path resembling the image under the left Dehn twist of a radius of the annulus in Figure~\ref{dehn twist} (see Figure~\ref{sewingTback}). (This switch from right to left Dehn twist may confuse the reader---precise and formal arguments are given in \cite[Section~2]{FGOP}.)

\begin{figure}[h!h!h!]
    \centerline{\includegraphics[width=1\textwidth]{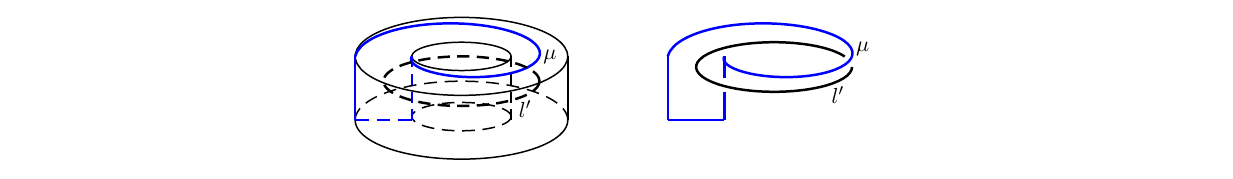}}
    \caption{Gluing the solid torus $\bar{\mathbb{T}}$ back following the attaching curve $\mu$}
\label{sewingTback}
\end{figure}

The left-hand side of Figure~\ref{framing1} illustrates the curve $l'$ (black), the attaching curve $\mu$ (blue) for the meridian and the red curve (the trace of the red dot in Figure~\ref{digging}), which is parallel both to $l$, and $l'$. At the right-hand side of Figure~\ref{framing1} the case when $\theta_l$ is replaced by $\theta_l^{-1}$ is illustrated. In the case of $\theta_l$, we conclude that the linking number of $l'$ and $\mu$ is by 1 greater than the self-linking of $l$. This linking number is the framing of the component $l'$ of the surgery.

\begin{figure}[h!h!h!]
    \centerline{\includegraphics[width=1\textwidth]{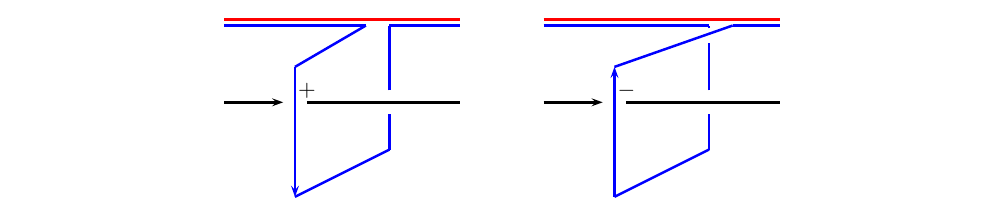}}
    \caption{Winding of the attaching curve}
\label{framing1}
\end{figure}

In the case when $h$ is a composition of several Dehn twists, we just repeat the above procedure save that we start from the rightmost Dehn twist, and the parallels of other curves are chosen deeper and deeper in $H$ (see \cite[paragraphs preceding Proposition~2.7]{FGOP} for a detailed explanation.)
\end{proof}

\section{An example}\label{exmpl}
The algorithm for finding an appropriate composition of Dehn twists and all the steps from above are illustrated in the following example.

Consider the manifold obtained by an identification of the facets of eight tetrahedra. Initially, the tetrahedra form an octahedron by identification of twelve pairs of facets. Then the ``opposite'' facets of this octahedron remain to be identified. This identification of the facets of the octahedron easily delivers the Heegaard diagram illustrated at the left-hand side of Figure~\ref{HD}. The $\alpha$-curves in this diagram are denoted by 1, 4, 7 and 12, while the $\beta$-curves are denoted by 3, 6, 9, 10. The $\alpha'$-curves are introduced
according to Remark~\ref{s3 diagram}, and are denoted by 2, 5, 8, 11.

\begin{figure}[h!h!h!]
    \centering
    {\includegraphics[width=1\textwidth]{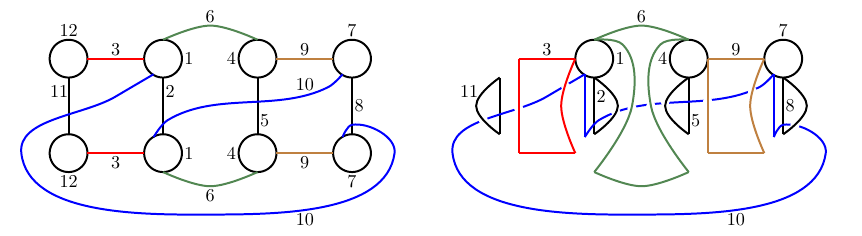}}
    \caption{Heegaard's diagram and the steps (S1) and (S2)}\label{HD}
\end{figure}

We start the algorithm for finding $\delta$ which maps the $\alpha'$-curves into the $\beta$-curves with $t=0$ and $\delta$ the identity (the empty composition of Dehn twists). The choice of matching $\alpha'$ and $\beta$-curves could simplify the procedure. Our choice is
$(\alpha'_1,\alpha'_2,\alpha'_3,\alpha'_4)=(2,5,8,11)$ and
$(\beta_1,\beta_2,\beta_3,\beta_4)=(3,6,9,10)$. With
$\beta_1'=\delta[\alpha'_1]=\alpha'_1=2$ we go to (A4). Luckily, for the
sequence of curves 2, 1, 3, the neighbours intersect in just one point and we have that the composition
$\theta_1\circ\theta_3\circ\theta_2\circ\theta_1$ maps the curve
$\alpha'_1=2$ into a curve isotopic to $\beta_1=3$. Then we go to (A1) with
$t=1$ and $\delta=\theta_1\circ\theta_3\circ\theta_2\circ\theta_1$.

With $\beta_2'=\delta[\alpha'_2]=\alpha'_2=5$ we go to (A4), and again in the
sequence of curves 5, 4, 6, which are disjoint from 3, the neighbours
intersect in just one point and we have that the composition
$\theta_4\circ\theta_6\circ\theta_5\circ\theta_4$ maps the curve
$\delta[\alpha'_2]=5$ into a curve isotopic to $\beta_2=6$. Then we go to (A1) with $t=2$ and $\delta=
\theta_4\circ\theta_6\circ\theta_5\circ\theta_4\circ\theta_1\circ
\theta_3\circ\theta_2\circ\theta_1$.

With $\beta_3'=\delta[\alpha'_3]=\alpha'_3=8$ we go to (A4), and again in the sequence of curves 8, 7, 9, which are disjoint from 3 and 6, the
neighbours intersect in just one point and we have that the composition
$\theta_7\circ\theta_9\circ\theta_8\circ\theta_7$ maps the curve
$\delta[\alpha'_3]=8$ into a curve isotopic to $\beta_3=9$. Then we go to (A1) with $t=3$ and $\delta=
\theta_7\circ\theta_9\circ\theta_8\circ\theta_7\circ
\theta_4\circ\theta_6\circ\theta_5\circ\theta_4\circ\theta_1\circ
\theta_3\circ\theta_2\circ\theta_1$.

With $\beta_4'=\delta[\alpha'_4]=\alpha'_4=11$ we go to (A5) since this curve intersects $\beta_4=10$ in a single point. The curves 10 and 11 are disjoint from 3, 6 and 9, and the composition
$\theta_{11}\circ\theta_{10}$ maps the curve $\delta[\alpha'_4]=11$ into a curve isotopic to $\beta_4=10$. Hence, the following composition of Dehn twists maps $\alpha'$-curves onto $\beta$-curves
\begin{equation*}
\theta_{11}\circ\theta_{10}\circ
\theta_7\circ\theta_9\circ\theta_8\circ\theta_7\circ
\theta_4\circ\theta_6\circ\theta_5\circ\theta_4\circ\theta_1\circ
\theta_3\circ\theta_2\circ\theta_1.
\end{equation*}

The steps  (S1) and (S2) in transforming our Heegaard diagram into a surgery are illustrated in Figure~\ref{HD}. In the step (S3) each curve at the right-hand side of Figure~\ref{HD}, except the blue one whose framing is 3, obtains the framing equal to 1. However, we will keep the numbers of curves until the last figure, where the framing is present. This is done in order to help the reader to follow the example. The step (S4) is illustrated by several pictures. At the left-hand side of Figure~\ref{HDa} we illustrate a parallel copy in $H$ of $\delta_2$, which is the curve denoted by 2 in our case. It is linked with the parallel copy of the curve denoted by 1 since it is pushed deeper into $H$. Then we make a parallel copy of $\delta_3$, which is the curve denoted by 3 in our case (see the right-hand side of Figure~\ref{HDa}).

\begin{figure}[h!h!h!]
    \centering
    \includegraphics[width=1\textwidth]{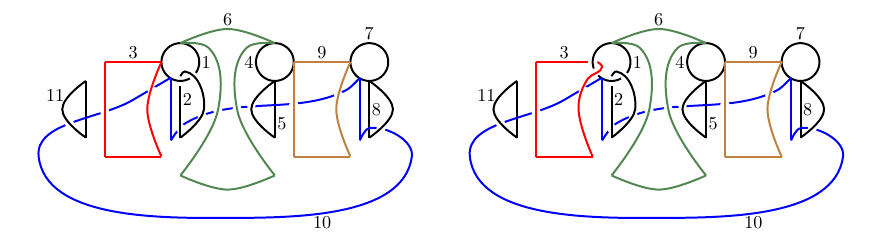}
    \caption{Step (S4)}\label{HDa}
\end{figure}

Next comes another parallel copy of the curve denoted by 1 inside a handle of $H$ (see the left-hand side of Figure~\ref{HDb}). We proceed in the same manner in order to make parallel copies of $\delta_5$, $\delta_6$ and $\delta_7$, which are the curves denoted by 4, 5 and 6 in our case (see the right-hand side of Figure~\ref{HDb}).
\begin{figure}[h!h!h!]
    \centering
    \includegraphics[width=1\textwidth]{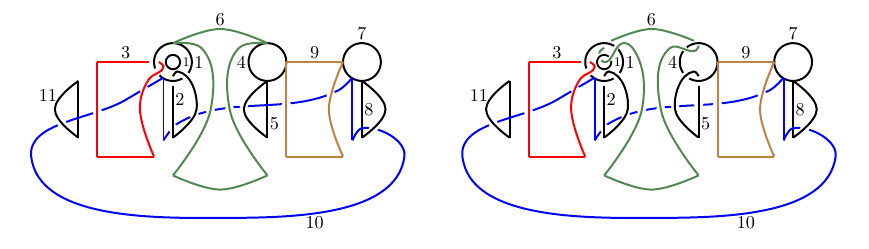}
    \caption{Step (S4) continued}\label{HDb}
\end{figure}

Eventually, we obtain a surgery diagram (with framing) illustrated in Figure~\ref{HDc}. It is straightforward, by using Kirby's calculus (see \cite{K78} and \cite{FR79}), to transform this surgery into an unknot with framing 2, justifying that the manifold in question is $RP^3$.

\begin{figure}[h!h!h!]
	\centering
    \includegraphics[width=1\textwidth]{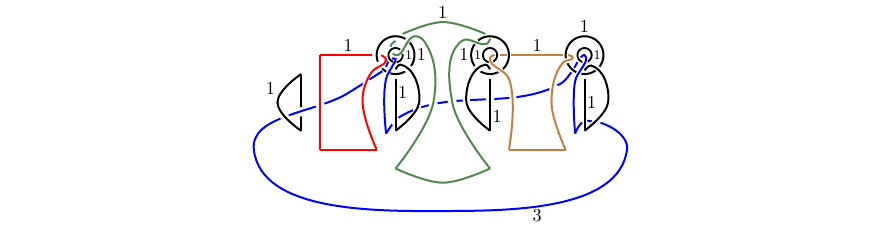}
    \caption{A surgery diagram for $RP^3$}\label{HDc}
\end{figure}

\textmd{\textbf{Acknowledgements.}}
We are grateful to the anonymous referee for helping us to improve the paper. We are also grateful to Vladimir Gruji\' c for pointing out to us the problem
of surgery presentation of small covers and to Danica Kosanovi\' c for
very useful comments concerning the previous version of this note.

\section*{Declarations}
\noindent\textbf{Funding} J.N. and Z.P. were supported by the Science Fund of the Republic of Serbia, Grant No. 7749891, Graphical Languages - GWORDS.

\noindent\textbf{Competing interests} The authors have no relevant financial or non-financial interests to disclose.

\end{document}